\documentclass[12pt]{article}
\usepackage{enumerate}
\usepackage{amssymb,amsmath,amsfonts,amsthm,mathrsfs}
\usepackage[colorlinks=true, pdfstartview=FitV, linkcolor=blue, citecolor=blue, urlcolor=blue]{hyperref}
\usepackage{graphicx,color}
\usepackage{cite}
\usepackage{indentfirst}

\allowdisplaybreaks

\begin{document}
\def\rn{{\mathbb R^n}}  \def\sn{{\mathbb S^{n-1}}}
\def\co{{\mathcal C_\Omega}}
\def\z{{\mathbb Z}}
\def\nm{{\mathbb (\rn)^m}}
\def\mm{{\mathbb (\rn)^{m+1}}}
\def\n{{\mathbb N}}
\def\cc{{\mathbb C}}

\newtheorem{defn}{Definition}
\newtheorem{thm}{Theorem}
\newtheorem{lem}{Lemma}
\newtheorem{cor}{Corollary}
\newtheorem{rem}{Remark}

\title{\bf\Large Mixed radial-angular bounds for Hardy-type operators on Heisenberg groups
\footnotetext{{\it Key words and phrases}: Hardy operator; dual operator; weighted Hardy operator; weighted Ces\`{a}ro operator; mixed radial-angular space; Heisenberg group.
\newline\indent\hspace{1mm} {\it 2020 Mathematics Subject Classification}: Primary 42B25; Secondary 42B20, 47H60, 47B47.}}

\date{}
\author{Zhongci Hang, Xiang Li\footnote{Corresponding author}, Dunyan Yan}
\maketitle
\begin{center}
\begin{minipage}{13cm}
{\small {\bf Abstract:}\quad
In this paper, we will study $n$-dimensional Hardy operator and its dual in mixed radial-angular spaces on Heisenberg groups and obtain their sharp bounds by using the rotation method. Furthermore, the sharp bounds of $n$-dimensional weighted Hardy operator and weighted Ces\`{a}ro operator are also obtained.}
\end{minipage}
\end{center}

\section{Introduction}\label{sec1}
\par
The classic Hardy operator and its dual operator are defined by
$$
H(f)(x):=\frac{1}{x}\int_{0}^{x}f(y)d y,\quad H^*(f)(x):=\int_{x}^{\infty}\frac{f(y)}{y}dy,
$$
for the locally integrable function $f$ on $\mathbb{R}$ and $x\neq0.$
The classic Hardy operator was introduced by Hardy \cite{Hardy} and he showed the following Hardy inequalities
$$
\|H(f)\|_{L^p}\leq\frac{p}{p-1}\|f\|_{L^p},\quad \|H^*(f)\|_{L^p}\leq p\|f\|_{L^p},
$$
 where $1<p<\infty$, the constants $\frac{p}{p-1}$, $p$ are best possible.
 
Faris\cite{Faris} first extended Hardy-type operator to higher dimension, Christ and Grafakos\cite{Christ} gave an equivalent version of $n$-dimensional Hardy operator $\mathcal{H}$ for nonnegative function $f$ on $\mathbb{R}^n$,
$$
\mathcal{H}f(x):=\frac{1}{\Omega_n|x|^n}\int_{|y|<|x|}f(y)dy,x\in\mathbb{R}^n\backslash\{0\},
$$
where $\Omega_n=\frac{\pi^{\frac{n}{2}}}{\Gamma(1+\frac{n}{2})}$ is the volume of the unit ball in $\mathbb{R}^n$.By direct computation,the dual operator of $\mathcal{H}$ can be defined by setting, for nonnegative function $f$ on $\mathbb{R}^n$,
$$
\mathcal{H}^*(f)(x):=\int_{|y|\geq|x|}\frac{f(y)}{\Omega_n|y|^n}dy,x\in\mathbb{R}^n\backslash\{0\}.
$$
 Christ and Grafakos\cite{Christ} proved that the norms of $\mathcal{H}$ and $\mathcal{H}^*$ on $L^p(\mathbb{R}^n)(1<p<\infty)$ are also $\frac{p}{p-1}$ and $p$, which is the same as that in the 1-dimensional case and is also independent of the dimension. The sharp weak estimate for $\mathcal{H}$ was obtained by Zhao et al.\cite{Zhao}.For $1\leq p\leq\infty$, 
 $$
 \|\mathcal{H}(f)\|_{L^{p,\infty}}\leq1\times\|f\|_{L^p},
 $$
 where 1 is best constant.
 
 In recent years, the research on Hardy operator related problems is receiving increasing attention. In \cite{Hardy2} , Hardy et al. provided us with the early development and application of    Hardy's inequalities.
 
 In this paper, we will investigate the sharp bound for Hardy-type operators in the setting of the  Heisenberg group, which plays an important role in several branches of mathematics. Now, let us to introduce some basic knowledge about the Heisenberg group which will be used in the following. The Heisenberg group $\mathbb{H}^n$ is a non-commutative nilpotent Lie group, with the underlying manifold $\mathbb{R}^{2n}\times\mathbb{R}$ with the group law
 $$
 x \circ y=\left(x_1+y_1, \ldots, x_{2 n}+y_{2 n}, x_{2 n+1}+y_{2 n+1}+2 \sum_{j=1}^n\left(y_j x_{n+j}-x_j y_{n+j}\right)\right)
 $$
 and
 $$
 \delta_r\left(x_1, x_2, \ldots, x_{2 n}, x_{2 n+1}\right)=\left(r x_1, r x_2, \ldots, r x_{2 n}, r^2 x_{2 n+1}\right), \quad r>0,
 $$
 where $x=(x_1,\cdots,x_{2n},x_{2n+1})$ and $y=(y_1,\cdots,y_{2n},y_{2n+1})$.
 The Haar measure on $\mathbb{H}^n$ coincides with the usual Lebesgue measure on $\mathbb{R}^{2n+1}$. We denote the measure of any measurable set $E \subset \mathbb{H}^n$ by $|E|$.Then
 $$
 |\delta_r(E)|=r^Q|E|, d(\delta_r x)=r_Q d x,
 $$
 where $Q=2n+2$ is called the homogeneous dimension of $\mathbb{H}^n$.
 
 The Heisenberg distance derived from the norm
 $$
 |x|_h=\left[\left(\sum_{i=1}^{2 n} x_i^2\right)^2+x_{2 n+1}^2\right]^{1 / 4},
 $$
 where $x=(x_1,x_2,\cdots,x_{2n},x_{2n+1})$,is given by
 $$
 d(p, q)=d\left(q^{-1} p, 0\right)=\left|q^{-1} p\right|_h.
 $$
 This distance $d$ is left-invariant in the sense that $d(p,q)$ remains unchanged when $p$ and $q$ are both left-translated by some fixed vector on $\mathbb{H}^n$.Furthermore,$d$ satisfies the triangular inequality\cite{AK}
 $$
 d(p, q) \leq d(p, x)+d(x, q), \quad p, x, q \in \mathbb{H}^n.
 $$
 For $r>0$ and $x\in\mathbb{H}^n$,the ball and sphere with center$x$ and radius $r$ on $\mathbb{H}^n$ are given by
 $$
 B(x, r)=\left\{y \in \mathbb{H}^n: d(x, y)<r\right\},
 $$
 and
 $$
 S(x, r)=\left\{y \in \mathbb{H}^n: d(x, y)=r\right\},
 $$
 respectively.And we have
 $$
 |B(x, r)|=|B(0, r)|=\Omega_Q r^Q,
 $$
 where
 $$
 \Omega_Q=\frac{2 \pi^{n+\frac{1}{2}} \Gamma(n / 2)}{(n+1) \Gamma(n) \Gamma((n+1) / 2)}
 $$
 is the volume of the unit ball $B(0,1)$ on $\mathbb{H}^n$, $\omega_Q=Q\Omega_Q$ (see\cite{CT}).
     More about Heisenberg group can refer to \cite{Zhao2}, \cite{GB} and \cite{ST}.
   
     The $n$-dimensional Hardy operator and its dual operator on Heisenberg group is defined by Wu and Fu \cite{Fu}
     \begin{equation}
     	H_hf(x):=\frac{1}{\Omega_Q|x|^Q_h}\int_{|y|_h<|x|_h}f (y)dy, H_h*f(x):=\int_{|y|_h\geq|x|_h}\frac{f(y)}{\Omega_Q|y|_h^Q}dy,
     \end{equation}
     where $x\in\mathbb{R}^n\backslash\{0\}$, $f$ be a locally integrable function on $\mathbb{H}^n$. They proved $\mathcal{H}_h$ and $\mathcal{H}_h^*$ is bounded from $L^p(\mathbb{H}^n)$ to $L^p(\mathbb{H}^n)$, $1<p\leq\infty$. Moreover,
     \begin{equation}
     \|\mathcal{H}_h\|_{L^p(\mathbb{H}^n)}=\frac{p}{p-1}\|f\|_{L^p(\mathbb{H}^n)}, \|\mathcal{H}^*_h\|_{L^p(\mathbb{H}^n)}=p\|f\|_{L^p(\mathbb{H}^n)}.
     \end{equation}
     This is the same as the result on $\mathbb{R}^n$.
     
     In \cite{Chu}, Chu et al. defined the $n$-dimensional weighted Hardy operator on Heisenberg group$\mathcal{H}_{hw}$ and $n$-dimensional weighted Ces$\grave{a}$ro operator on Heisenberg group $\mathcal{H}^*_{hw}$. Let us recall their definition.
     \begin{defn}
     	Let $w:[0,1]\rightarrow[0,\infty)$ be a function, for a measurable function $f$ on $\mathbb{H}^n$. The $n$-dimensional weighted Hardy operator on Heisenberg group $\mathcal{H}_{hw}$ is defined by 
     	$$
     	\mathcal{H}_{hw}:=\int_{0}^{1}f(\delta_t x)w(t)dt, x\in\mathbb{H}^n. 
     	$$
     \end{defn}   
     \begin{defn}
     	For a nonnegative function $w:[0,1]\rightarrow(0,\infty)$ and measurable complex-valued function $f$ on $\mathbb{H}^n$, the $n$-dimensional weighted  Ces\`{a}ro operator is defined by 
     	$$
     	\mathcal{H}^*_{hw}:=\int_{0}^{1}\frac{f(\delta_{1/t}x)}{t^Q}w(t)dt, x\in\mathbb{H}^n,
     	$$
     	which satisfies
     	$$
     	\int_{\mathbb{H}^n}f(x)(\mathcal{H}_{hw}g)(x)dx=\int_{\mathbb{H}^n}g(x)(\mathcal{H}^*_{hw})(x)dx,
     	$$
     	where $f\in L^p(\mathbb{H}^n), g\in L^q(\mathbb{H}^n),1<p<\infty,q=p/(p-1)$, $\mathcal{H}_{hw}$ is bounded on $L^p(\mathbb{H}^n)$, and $\mathcal{H}^*_{hw}$ is bounded on $L^q(\mathbb{H}^n)$.  
     \end{defn}
      Recently, many operators in harmonic analysis have been proved to be bounded on mixed radial-angular spaces, for instance, Duoandikoetxea and Oruetxebarria \cite{JD} built the extrapolation theorems on mixed radial-angular spaces to study the boundedness of a large class of operators which are weighted bounded. In \cite{WMQ}, Wei and Yan studied the sharp bounds for $n$-dimensional Hardy operator and its dual in mixed radial-angular spaces on Euclidean space. Inspired by them, we will  investigate the sharp bounds for $n$-dimensional Hardy operator and its dual operator in mixed radial-angular spaces on Heisenberg groups.
      
      Now, we give the definition of mixed radial-angular spaces on Heisenberg group.
      \begin{defn}
      	For any $n\geq2$,$1\leq p$,$\bar{p}\leq\infty$, the mixed radial-angular space $L^p_{|x|_h} L^{\bar{p}}_\theta(\mathbb{H}^n)$ consists of all functions $f$ in $\mathbb{H}^n$ for which 
      	$$
      	\|f\|_{L^p_{|x|}L^{\bar{p}}_\theta(\mathbb{H}^n)}:=\left(\int_0^\infty \left(\int_{\mathbb{S}^{Q-1}}|f(r,\theta)|^pd \theta\right)^{\frac{p}{\bar{p}}}r^{Q-1}dr\right)^{\frac{1}{p}}<\infty,
      	$$ 
      	where $S^{Q-1}$ denotes the unit sphere in $\mathbb{H}^n$.
      \end{defn} 
     Next, we will provide the main results of this article.
     \section{Mixed radial-angular bounds for $\mathcal{H}_h$ and $\mathcal{H}^*_h$.  }
\begin{thm}\label{thm1}
	Let $n\geq2,1<p, \bar{p}_1, \bar{p}_2<\infty$. Then $\mathcal{H}_h$ is bounded from $L_{|x|_h}^p L_\theta^{\bar{p}_1}(\mathbb{H}^n) $ to $ L_{|x|_h}^p L_\theta^{\bar{p}_2}(\mathbb{H}^n)$.Moreover,
	$$
	\|\mathcal{H}_h\|_{L_{|x|_h}^p L_\theta^{\bar{p}_1}(\mathbb{H}^n) \rightarrow L_{|x|_h}^p L_\theta^{\bar{p}_2}(\mathbb{H}^n)}=\left(\frac{pQ^{1/\bar{p}_2-1/\bar{p}_1}}{p-1}\right)\left(\frac{2 \pi^{n+\frac{1}{2}} \Gamma(n / 2)}{(n+1) \Gamma(n) \Gamma((n+1) / 2)}\right)^{1/\bar{p}_2-1/\bar{p}_1} .
	$$
\end{thm}
\begin{thm}\label{thm2}
	Let $n\geq2,1<p, \bar{p}_1, \bar{p}_2<\infty$, then $\mathcal{H}_h^*$ is bounded from $L_{|x|_h}^p L_\theta^{\bar{p}_1}(\mathbb{H}^n) $ to $ L_{|x|_h}^p L_\theta^{\bar{p}_2}(\mathbb{H}^n)$.Moreover,
	$$
	\|\mathcal{H}^*_h\|_{L_{|x|_h}^p L_\theta^{\bar{p}_1}(\mathbb{H}^n) \rightarrow L_{|x|_h}^p L_\theta^{\bar{p}_2}(\mathbb{H}^n)}= pQ^{1/\bar{p}_2-1/\bar{p}_1}\left(\frac{2 \pi^{n+\frac{1}{2}} \Gamma(n / 2)}{(n+1) \Gamma(n) \Gamma((n+1) / 2)}\right)^{1/\bar{p}_2-1/\bar{p}_1}.
	$$
\end{thm}
\begin{proof}[Proof of Theorem \ref{thm1} .]
	Set 
	\begin{equation}
		g(x)=\frac{1}{\omega_Q}\int_{\mathbb{S}^{Q-1}}f(\delta_{|x|_h}\theta)d\theta, x\in\mathbb{H}^n, 
	\end{equation}
	then $g$ is a radial function. Moreover, we have
	$$
	\begin{aligned}
		\|g\|_{L^p_{|x|_h}L^{\bar{p}_1}_\theta(\mathbb{H}^n)}&=\left(\int_0^\infty\left(\int_{\mathbb{S}^{Q-1}}|g(r,\theta)|^{\bar{p}_1}d\theta\right)^{p/\bar{p}_1}r^{Q-1}d r\right)^{1/p}\\
		&=\left(\int_0^\infty\left(\frac{1}{\omega_Q}|g(r)|^{\bar{p}_1}\right)^{p/\bar{p}_1}r^{Q-1}d r\right)^{1/p}\\
		&=\omega_Q^{1/\bar{p}_1}\left(\int_0^\infty|g(r)|^pr^{Q-1}d r\right)^{1/p},
		\end{aligned}
	$$
	where $g(r)$ can be defined as $g(r)=g(x)$ for any $x\in\mathbb{H}^n$ with $|x|_h=r$.
	By using H\"{o}lder inequality, we have 
	$$	
	\begin{aligned}
		\|g\|_{L^p_{|x|_h}L^{\bar{p}_1}_\theta(\mathbb{H}^n)}&=\omega_Q^{1/\bar{p}_1}\left(\int_{0}^{\infty}\left|\frac{1}{\omega_Q}\int_{S^{Q-1}}f(\delta_r\theta)d\theta\right|^pr^{Q-1}dr\right)^{1/p}\\
		&=\omega_Q^{1/\bar{p}_1-1}\left(\int_{0}^{\infty}\left|\int_{S^{Q-1}}f(\delta_r\theta)d\theta\right|^pr^{Q-1}dr\right)^{1/p}\\
		&\leq\omega_Q^{1/p-1}\left(\int_0^\infty\left(\int_{S^{Q-1}}|f(\delta_r\theta)|^{\bar{p}_1}d\theta\right)^{p/\bar{p}_1}\left(\int_{S^{Q-1}}d\theta\right)^{p/\bar{p}_1}r^{Q-1}dr\right)^{1/p}\\
		&=\left(\int_0^\infty\left(\int_{S^{Q-1}}|f(\delta_r\theta)|^{\bar{p}_1}d\theta\right)^{p/\bar{p}_1}r^{Q-1}dr\right)^{1/p}\\
		&=\|f\|_{L^p_{|x|_h}L^{\bar{p}_1}_\theta(\mathbb{H}^n)}.
	\end{aligned}
	$$ 
	By change of variables, we can get 
	$$
	\begin{aligned}
	\mathcal{H}_hg(x)&=\frac{1}{\Omega_Q|x|_h^Q}\int_{|y|_h<|x|_h}\left(\frac{1}{\omega_Q}\int_{\mathbb{S}^{Q-1}}f(\delta_{|x|_h\theta})d\theta\right)dy\\
	&=\frac{1}{\Omega_Q|x|_h^Q}\int_0^1\int_{S(0,1)}\left(\frac{1}{\omega_Q}\int_{\mathbb{S}^{Q-1}}f(\delta_{r}\theta)d\theta\right)r^{Q-1}dy^{'}dr\\
    &=\frac{1}{\Omega_Q|x|_h^Q}\int_0^1\int_{S(0,1)}f(\delta_r \theta)r^{Q-1}d\theta d r\\
    &=\mathcal{H}_hf(x).
	\end{aligned}
	$$
	Thus, we obtain
	$$
	\frac{\|\mathcal{H}_h(f)\|_{L_{|x|_h}^p L_\theta^{\bar{p}_2}(\mathbb{H}^n)}}{\|f\|_{L_{|x|_h}^p L_\theta^{\bar{p}_1}(\mathbb{H}^n)}} \leq \frac{\|\mathcal{H}_h(g)\|_{L_{|x|_h}^p L_\theta^{\bar{p}_2}(\mathbb{H}^n)}}{\|g\|_{L_{|x|_h}^p L_\theta^{\bar{p}_1}(\mathbb{H}^n)}}.
	$$
	This implies the operator $\mathcal{H}$ and its restriction to radial function have same norm from $L_{|x|_h}^p L_\theta^{\bar{p}_1}$ to $L_{|x|_h}^p L_\theta^{\bar{p}_2}$, without loss of generality, we can assume that $f$ is a radial function in the rest of the proof.
	Consequently, we have 
	$$
	\begin{aligned}
	\|\mathcal{H}_hf\|_{L_{|x|_h}^p L_\theta^{\bar{p}_2}(\mathbb{H}^n)}&=\left(\int_0^\infty\left(\int_{S^{Q-1}}|\mathcal{H}_h(f)(r,\theta)|^{\bar{p}_2}d\theta\right)^{p/\bar{p}_2}r^{Q-1}d r\right)^{\frac{1}{p}}\\
	&=\left(\int_0^\infty\left(\int_{S^{Q-1}}|\mathcal{H}_h(f)(r)|^{\bar{p}_2}d\theta\right)^{p/\bar{p}_2}r^{Q-1}d r\right)^{1/p}\\
	&=\omega_Q^{1/\bar{p}_2}\left(\int_0^\infty|\mathcal{H}_h(f)(r)|^pr^{Q-1}dr\right)^{1/p},
	\end{aligned}
	$$
	where $\mathcal{H}_h(f)(r)$ can be defined as $\mathcal{H}_h(f)(r)=\mathcal{H}_h(f)(x)$ for any $|x|_h=r$.
	Next, we use another form of Hardy operator 
	$$
	\mathcal{H}_h(f)(r)=\frac{1}{|B(0,r)|}\int_{B(0,r)}f( y)dy,x\in\mathbb{H}^n\backslash\{0\},
	$$
	by changing variables, we have
	$$
	\mathcal{H}_h(f)(r)=\frac{1}{\Omega_Q}\int_{B(0,1)}f(\delta_ry)dy,x\in\mathbb{H}^n\backslash\{0\}.
	$$
	Using Minkowski's inequality, we can get
	$$
	\begin{aligned}
	\|\mathcal{H}_hf\|_{L_{|x|_h}^p L_\theta^{\bar{p}_2}(\mathbb{H}^n)}&=\omega_Q^{1/\bar{p}_2}\left(\int_0^\infty\left|\frac{1}{\omega_Q}\int_{B(0,1)}f(\delta_{r}y)dy\right|^p r^{Q-1}dr\right)^{1/p}\\
	&=\frac{\omega_Q^{1/\bar{p}_2}}{\Omega_Q}\left(\int_0^\infty\left|\int_{B(0,1)}f(\delta_{r}y)dy\right|^p r^{Q-1}dr\right)^{1/p}\\
	&\leq\frac{\omega_Q^{1/\bar{P}_2}}{\Omega_Q}\int_{B(0,1)}\left(\int_0^\infty|f(\delta_{|y|_h}r )|^pr^{Q-1}dr\right)^{1/p}dy\\
	&=\frac{\omega_Q^{1/\bar{P}_2}}{\Omega_Q}\int_{B(0,1)}\left(\int_{0}^{\infty}|f(r)|^pr^{Q-1}dr\right)^{1/p}|y|_h^{-Q/p}dy\\
	&=\frac{\omega_Q^{1/\bar{p}_2-1/\bar{p}_1}}{\Omega_Q}\int_{B(0,1)}\left(\int_{0}^{\infty}\omega_Q^{p/\bar{p}_1}|f(r)|^pr^{Q-1}dr\right)^{1/p}|y|_h^{-Q/p}dy\\
	&=\frac{\omega_Q^{1/\bar{p}_2-1/\bar{p}_1}}{\Omega_Q}\int_{B(0,1)}|y|_h^{-Q/p}dy\|f\|_{L^p_{|x|_h}L^{\bar{p}_1}_\theta}\\
	&=\frac{p}{p-1}\omega_Q^{1/\bar{p}_2-1/\bar{p}_1}\|f\|_{L^p_{|x|_h}L^{\bar{p}_1}_\theta}\\
	&=\left(\frac{pQ^{1/\bar{p}_2-1/\bar{p}_1}}{p-1}\right)\left(\frac{2 \pi^{n+\frac{1}{2}} \Gamma(n / 2)}{(n+1) \Gamma(n) \Gamma((n+1) / 2)}\right)^{1/\bar{p}_2-1/\bar{p}_1}\|f\|_{L^p_{|x|_h}L^{\bar{p}_1}_\theta}.
	\end{aligned}
	$$
	Therefore, we have
	\begin{equation}\label{leq}
		\|\mathcal{H}_hf\|_{L_{|x|_h}^p L_\theta^{\bar{p}_2}(\mathbb{H}^n)}\leq\left(\frac{pQ^{1/\bar{p}_2-1/\bar{p}_1}}{p-1}\right)\left(\frac{2 \pi^{n+\frac{1}{2}} \Gamma(n / 2)}{(n+1) \Gamma(n) \Gamma((n+1) / 2)}\right)^{1/\bar{p}_2-1/\bar{p}_1}\|f\|_{L^p_{|x|_h}L^{\bar{p}_1}_\theta}.
	\end{equation}
	On the other hand, for $0<\epsilon<1$, take 
	$$
	f_\epsilon(x)= \begin{cases}0, & |x|_h \leq 1, \\ |x|_h^{-\left(\frac{Q}{p}+\epsilon\right)} & |x|_h>1\end{cases},
	$$
	then
	$$
	\|f_\epsilon\|^p_{L^p_{|x|_h}L^{\bar{p}_1}_\theta}=\frac{\omega_Q^{p/\bar{p}_1}}{p\epsilon},
	$$
	and
	$$
	\mathcal{H}_h(f_\epsilon)(x)= \begin{cases}0, & |x|_h \leq 1, \\ \Omega_Q^{-1}|x|_h^{-{\frac{Q}{p}}-\epsilon} \int_{|x|_h^{-1}<|y|_h <1}|y|_h^{-\frac{Q}{p}-\epsilon} d y, & |x|_h>1\end{cases}.
	$$
	So, we have
	$$
	\begin{aligned}
	\|\mathcal{H}_h(f_\epsilon)\|_{L^p_{|x|_h}L^{\bar{p}_2}_\theta(\mathbb{H}^n)}&=\frac{\omega_Q^{1/\bar{p}_2}}{\Omega_Q}\left(\int_{r>1}\left|r^{-\frac{Q}{p}-\epsilon}\int_{r^{-1}<|y|_h<1}|y|_h^{-\frac{Q}{p}-\epsilon}dy\right|^pr^{Q-1}dr\right)^{1/p}\\
	&\geq\frac{\omega_Q^{1/\bar{p}_2}}{\Omega_Q}\left(\int_{r>\frac{1}{\epsilon}}\left|r^{-\frac{Q}{p}-\epsilon}\int_{\epsilon<|y|_h<1}|y|_h^{-\frac{Q}{p}-\epsilon}dy\right|^pr^{Q-1}dr\right)^{1/p}\\
	&=\frac{\omega_Q^{1/\bar{p}_2}}{\Omega_Q}\left(\int_{r>\frac{1}{\epsilon}}r^{-p\epsilon-Q}dr\right)^{1/p}\int_{\epsilon<|y|_h<1}|y|_h^{-\frac{Q}{p}-\epsilon}dy\\
	&=\frac{\omega_Q^{1+1/\bar{p}_2}}{\Omega_Q}\left(\int_{r>\frac{1}{\epsilon}}r^{-p\epsilon-Q}dr\right)^{1/p}\int_{1}^{\epsilon}r^{Q-1-\frac{Q}{p}-\epsilon}dr\\
	&=\epsilon^\epsilon\frac{1-\epsilon^{Q-\frac{Q}{p}-\epsilon}}{1-\frac{1}{p}-\frac{\epsilon}{Q}}\omega_Q^{1/\bar{p}_2-1/\bar{p}_1}\|f_\epsilon\|_{L^p_{|x|_h}L^{\bar{p}_1}_\theta}.
	\end{aligned}
	$$
	Thus, we have obtained
	$$
\|\mathcal{H}_h\|_{L_{|x|_h}^p L_\theta^{\bar{p}_1(\mathbb{H}^n)} \rightarrow L_{|x|_h}^p L_\theta^{\bar{p}_2}(\mathbb{H}^n)}\geq\epsilon^\epsilon\frac{1-\epsilon^{Q-\frac{Q}{p}-\epsilon}}{1-\frac{1}{p}-\frac{\epsilon}{Q}}\omega_Q^{1/\bar{p}_2-1/\bar{p}_1}\|f_\epsilon\|_{L^p_{|x|_h}L^{\bar{p}_1}_\theta}.
	$$
	Since $\epsilon^\epsilon\rightarrow1$ as $\epsilon\rightarrow0$, by letting $\epsilon\rightarrow0$, we have
	\begin{equation}\label{geq}
		\begin{aligned}
\|\mathcal{H}_h\|_{L_{|x|_h}^p L_\theta^{\bar{p}_1}(\mathbb{H}^n}&\geq\frac{p}{p-1}\omega_Q^{1/\bar{p}_2-1/\bar{p}_1}\\
&=\left(\frac{pQ^{1/\bar{p}_2-1/\bar{p}_1}}{p-1}\right)\left(\frac{2 \pi^{n+\frac{1}{2}} \Gamma(n / 2)}{(n+1) \Gamma(n) \Gamma((n+1) / 2)}\right)^{1/\bar{p}_2-1/\bar{p}_1}\|f\|_{L^p_{|x|_h}L^{\bar{p}_1}_\theta}.
         \end{aligned}	
	\end{equation}
	Combine (\ref{leq}) and (\ref{geq}), we can get 
	$$
		\|\mathcal{H}_hf\|_{L_{|x|_h}^p L_\theta^{\bar{p}_2}(\mathbb{H}^n)}=\left(\frac{pQ^{1/\bar{p}_2-1/\bar{p}_1}}{p-1}\right)\left(\frac{2 \pi^{n+\frac{1}{2}} \Gamma(n / 2)}{(n+1) \Gamma(n) \Gamma((n+1) / 2)}\right)^{1/\bar{p}_2-1/\bar{p}_1}\|f\|_{L^p_{|x|_h}L^{\bar{p}_1}_\theta}.
	$$
	This completes the proof of Theorem \ref{thm1}.
	\end{proof}
	\begin{proof}[Proof of Theorem \ref{thm2}.]
		The proof of Theorem \ref{thm2} is similar to prove of Theorem \ref{thm1}, we omit the details.
	\end{proof}
\section{Mixed radial-angular bounds for $\mathcal{H}_{hw}$ and $\mathcal{H}^*_{hw}$.} 
\begin{thm}\label{thm3}
Let $w:[0,1]\rightarrow(0,\infty)$ be a function, $n\leq2, 1<p,\bar{p}_1,\bar{p}_2<\infty$. Then the $n$-dimensional weighted Hardy operator on Heisenberg group $\mathcal{H}_{hw}$ is bounded from $L_{|x|_h}^p L_\theta^{\bar{p}_1}(\mathbb{H}^n) $ to $ L_{|x|_h}^p L_\theta^{\bar{p}_2}(\mathbb{H}^n)$. Moreover,
$$\begin{aligned}
\|\mathcal{H}_{hw}\|_{L_{|x|_h}^p L_\theta^{\bar{p}_1}(\mathbb{H}^n) \rightarrow L_{|x|_h}^p L_\theta^{\bar{p}_2}(\mathbb{H}^n)}=& Q^{1/\bar{p}_2-1/\bar{p}_1}\left(\frac{2 \pi^{n+\frac{1}{2}} \Gamma(n / 2)}{(n+1) \Gamma(n) \Gamma((n+1) / 2)}\right)^{1/\bar{p}_2-1/\bar{p}_1}\\
&\times\int_{0}^{1}t^{-\frac{Q}{p}}w(t)dt.
\end{aligned}$$
\end{thm}
\begin{thm}\label{thm4}
Let $w:[0,1]\rightarrow(0,\infty)$ be a function, $n\leq2, 1<p,\bar{p}_1,\bar{p}_2<\infty$. Then the $n$-dimensional weighted Ces$\grave{a}$ro operator on Heisenberg group $\mathcal{H}^*_{hw}$ is bounded from $L_{|x|_h}^p L_\theta^{\bar{p}_1}(\mathbb{H}^n) $ to $ L_{|x|_h}^p L_\theta^{\bar{p}_2}(\mathbb{H}^n)$. Moreover,
$$\begin{aligned}
\|\mathcal{H}^*_{hw}\|_{L_{|x|_h}^p L_\theta^{\bar{p}_1}(\mathbb{H}^n) \rightarrow L_{|x|_h}^p L_\theta^{\bar{p}_2}(\mathbb{H}^n)}=& Q^{1/\bar{p}_2-1/\bar{p}_1}\left(\frac{2 \pi^{n+\frac{1}{2}} \Gamma(n / 2)}{(n+1) \Gamma(n) \Gamma((n+1) / 2)}\right)^{1/\bar{p}_2-1/\bar{p}_1}\\
&\times\int_{0}^{1}t^{-Q(1-\frac{1}{p})}w(t)dt.
\end{aligned}$$
\end{thm}
The proof methods for Theorem \ref{thm3} and Theorem \ref{thm4} are the same, and similar to the proof method for Theorem \ref{thm1} . But as a special case, here we will give the proof of Theorem \ref{thm4}.
\begin{proof}[Proof of Theorem \ref{thm4}.]
Inspired by proof of Theorem \ref{thm1}, we have
$$
\|\mathcal{H}^*_{hw}\|_{L^p_{|x|_h}L^{\bar{p}_2}_\theta(\mathbb{H}^n)}=\omega_Q^{1/\bar{p}_2}\left(\int_{0}^{\infty}|\mathcal{H}_{hw}(f)(r)|^pr^{Q-1}dr\right)^{1/p},
$$
where $\mathcal{H}^*_{hw}(f)(r)$ can be defined as $\mathcal{H}^*_{hw}(f)(r)=\mathcal{H}^*_{hw}(f)(x)$ for any $|x|_h=r$. Using Minkowski's inequality, we can get that
$$
\begin{aligned}
\|\mathcal{H}^*_{hw}\|_{L^P_{|x|_h}L^{\bar{p}_2}_\theta(\mathbb{H}^n)}=&\omega_Q^{1/\bar{p}_2}\left(\int_{0}^{\infty}\left|\int_{0}^{1}\frac{f(\delta_{1/r}t)}{t^Q}w(t)dt\right|^pr^{Q-1}dr\right)^{1/p}\\
\leq&\omega_Q^{1/\bar{p}_2}\int_{0}^{1}\left(\int_{0}^{\infty}|f(\delta_{1/t}r)|^pr^{Q-1}dr\right)^{1/p}t^{-Q}w(t)dt\\
=&\omega_Q^{1/\bar{p}_2}\int_0^1\left(\int_0^\infty|f(r)|^pr^{Q-1}dr\right)^{1/p}t^{-Q+Q/p}w(t)dt\\
=&\omega_Q^{1/\bar{p}_2-1/\bar{p}_1}\int_0^1\left(\int_0^\infty\omega_Q^{p/\bar{p}_1}|f(r)|^pr^{Q-1}dr\right)^{1/p}t^{-Q+Q/p}w(t)dt\\
=&\omega_Q^{1/\bar{p}_2-1/\bar{p}_1}\int_0^1t^{-Q(1-\frac{1}{p})}w(t)dt\|f\|_{L^p_{|x|_h}L^{\bar{p}_1}_\theta}\\
=& Q^{1/\bar{p}_2-1/\bar{p}_1}\left(\frac{2 \pi^{n+\frac{1}{2}} \Gamma(n / 2)}{(n+1) \Gamma(n) \Gamma((n+1) / 2)}\right)^{1/\bar{p}_2-1/\bar{p}_1}\\
\end{aligned}
$$
$$
\begin{aligned}
&\times\int_{0}^{1}t^{-Q(1-\frac{1}{p})}w(t)dt\|f\|_{L^p_{|x|_h}L^{\bar{p}_1}_\theta}.
\end{aligned}
$$
Therefore, we have
$$\begin{aligned}
\|\mathcal{H}^*_{hw}\|_{L^p_{|x|_h}L^{\bar{p}_2}_\theta(\mathbb{H}^n)}\leq & Q^{1/\bar{p}_2-1/\bar{p}_1}\left(\frac{2 \pi^{n+\frac{1}{2}} \Gamma(n / 2)}{(n+1) \Gamma(n) \Gamma((n+1) / 2)}\right)^{1/\bar{p}_2-1/\bar{p}_1}\\
&\times\int_{0}^{1}t^{-Q(1-\frac{1}{p})}w(t)dt\|f\|_{L^p_{|x|_h}L^{\bar{p}_1}_\theta}.
\end{aligned}
$$
On the other hand,  taking
$$
C=\|\mathcal{H}^*_{hw}\|_{L^p_{|x|_h}L^{\bar{p}_2}_\theta(\mathbb{H}^n)\rightarrow L^p_{|x|_h}L^{\bar{p}_1}_\theta(\mathbb{H}^n)}<\infty,
$$
and for $f\in L^{p}_{|x|_h}L^{\bar{p}_2}_\theta(\mathbb{H}^n)$, we obtain
$$
\|\mathcal{H}^*_{hw}\|_{L^p_{|x|_h}L^{\bar{p}_2}_\theta(\mathbb{H}^n)}\leq C\|f\|L^p_{|x|_h}L^{\bar{p}_1}_\theta(\mathbb{H}^n).
$$
For any $\epsilon>0$, take
$$
f_\epsilon(t)= \begin{cases}0, & t \leq 1 ,\\ t^{-\left(\frac{Q}{p}+\epsilon\right)} & t>1.\end{cases}
$$
Then
$$
\|f_\epsilon\|^p_{L^p_{|x|_h}L^{\bar{p}_1}_\theta}=\frac{\omega_Q^{p/\bar{p}_1}}{p\epsilon},
$$
and
$$
\mathcal{H}^*_{hw}(f_\epsilon)(x)= \begin{cases}0, & |x|_h \leq 1, \\ |x|_h^{-{\frac{Q}{p}}-\epsilon} \int_{|x|_h^{-1}<t <1}t^{\frac{Q}{p}+\epsilon-Q}w(t) d t, & |x|_h>1\end{cases}.
$$
So we have 
$$
\begin{aligned}
C^p\|f_\epsilon\|^P_{L^p_{|x|_h}L^{\bar{p}_1}_\theta}&\geq\|\mathcal{H}^*_{hw}\|^p_{L^p_{|x|_h}L^{\bar{p}_2}_\theta}\\
&=\omega_Q^{1/\bar{p}_2}\left(\int_{r>1}\left|r^{-\frac{Q}{p}-\epsilon}\int_{r^{-1}<t<1}t^{\frac{Q}{p}+\epsilon-Q}dt\right|^pr^{Q-1} dr\right)^{1/p}\\
&\geq\omega_Q^{1/\bar{p}_2}\left(\int_{r>\frac{1}{\epsilon}}\left|r^{-\frac{Q}{p}-\epsilon}\int_{\epsilon<t<1}t^{\frac{Q}{p}+\epsilon-Q}w(t)dt\right|^pr^{Q-1}dr\right)^{1/p}\\
&=\omega_Q^{1/\bar{p}_2}\int_{r>\frac{1}{\epsilon}}r^{-p\epsilon-Q}dr\left(\int_{\epsilon<t<1}t^{\frac{Q}{p}+\epsilon-Q}w(t)dt\right)^p.
\end{aligned}
$$
By changing of variable $r=\delta_{1/\epsilon}y$, we have
$$
\begin{aligned}
C^p\|f_\epsilon\|^P_{L^p_{|x|_h}L^{\bar{p}_1}}&\geq\omega_Q^{1/\bar{p}_2}\int_{|y|_h>1}|y|_h^{-p\epsilon-Q}\epsilon^{\epsilon p}dy\left(\int_{\epsilon<t<1}t^{\frac{Q}{p}+\epsilon-Q}w(t)dt\right)^p\\
&=\omega_Q^{1/\bar{p}_2-1/\bar{p}_1}\left(\epsilon^\epsilon\int_{1<t<\epsilon}t^{\frac{Q}{p}+\epsilon-Q}w(t)dt\right)^p\|f_\epsilon\|_{L^P_{|x|_h}L^{\bar{p}_1}_\theta(\mathbb{H}^n)}.
\end{aligned}
$$ 
This implies that
$$
\epsilon^\epsilon\int_{1<t<\epsilon}t^{\frac{Q}{p}+\epsilon-Q}w(t)dt\leq C.
$$
Let $\epsilon\rightarrow0$, we have
$$
\int_{0}^1t^{\frac{Q}{p}-Q}w(t)dt\leq C.
$$
Thus, we have finished the proof of Theorem \ref{thm4}.
\end{proof}
\par
\subsection*{Acknowledgements}
This work was supported by National Natural Science Foundation of  China (Grant No. 12271232) and Shandong Jianzhu University Foundation (Grant No. X20075Z0101).

\begin{flushleft}
	
	\vspace{0.3cm}\textsc{Zhongci Hang\\School of Science\\Shandong Jianzhu University \\Jinan, 250000\\P. R. China}

    \emph{E-mail address}: \textsf{babysbreath4fc4@163.com}
	
		\vspace{0.3cm}\textsc{Xiang Li\\School of Science\\Shandong Jianzhu University\\Jinan, 250000\\P. R. China}
	
	\emph{E-mail address}: \textsf{lixiang162@mails.ucas.ac.cn}

	\vspace{0.3cm}\textsc{Dunyan Yan\\School of Mathematical Sciences\\University of Chinese Academy of Sciences\\Beijing, 100049\\P. R. China}
	
	\emph{E-mail address}: \textsf{ydunyan@ucas.ac.cn}
	
\end{flushleft}

\end{document}